%% file: ms.tex
\documentclass[11pt]{article}

\usepackage{amsfonts}
\usepackage{amsmath}
\usepackage{amsthm}
\usepackage{dsfont}
\usepackage{hyperref}
\usepackage[capitalise]{cleveref}
\usepackage{listings}
\usepackage{mdframed}
\usepackage{pgfplots}
\usepackage{subcaption}
\usepackage{xcolor}

\newcommand{\footremember}[2]{%
    \footnote{#2}
    \newcounter{#1}
    \setcounter{#1}{\value{footnote}}%
}

\newtheoremstyle{break}
  {\topsep}{\topsep}%
  {\itshape}{}%
  {\bfseries}{}%
  {\newline}{}%
\theoremstyle{break}
\newtheorem{theorem}{Theorem}
\newtheorem{prop}{Proposition}

\newtheorem{definition}{Definition}

\definecolor{light-gray}{gray}{0.95}

\begin{document}

\title{High-performance sampling of generic Determinantal
Point Processes}
\author{Jack Poulson\footremember{hodge}{
  jack@hodgestar.com, Hodge Star Scientific Computing}}

\maketitle


\input{abstract}
\input{intro}
\input{prototype_factorization}
\input{hpc_dense_factorization}
\input{sparse_direct_factorization}
\input{conclusions}
\input{reproducibility}
\input{acknowledgement}

\bibliography{hpc_dpp}

\end{document}

%% file: abstract.tex
\begin{abstract}
Determinantal Point Processes (DPPs) were introduced by
Macchi~\cite{Macchi-1975} as a model for repulsive (fermionic) particle
distributions. But their recent popularization is largely due to their
usefulness for encouraging diversity in the final stage of a recommender
system~\cite{KuleszaTaskar-2012}.

The standard sampling scheme for finite DPPs is a spectral decomposition
followed by an equivalent of a randomly diagonally-pivoted Cholesky
factorization of an orthogonal projection, which
is only applicable to Hermitian kernels and has an expensive setup cost.
Researchers~\cite{Launay-2018,ChenEtAl-2018} have begun to connect DPP sampling
to $LDL^H$ factorizations as a
means of avoiding the initial spectral decomposition, but existing approaches
have only outperformed the spectral decomposition approach in special circumstances, where the number of kept modes is a small percentage of the ground set
size.

This article proves that trivial modifications of $LU$ and $LDL^H$
factorizations yield efficient direct sampling schemes for non-Hermitian and
Hermitian DPP kernels, respectively. Further, it is experimentally shown that
even dynamically-scheduled, shared-memory parallelizations of high-performance
dense and sparse-direct factorizations can be trivially modified to yield DPP
sampling schemes with essentially identical performance.


The software developed as part of this research, Catamari
[\href{https://hodgestar.com/catamari}{hodgestar.com/catamari}] is released
under the Mozilla Public License v2.0. It contains header-only, C++14 plus
OpenMP 4.0 implementations of dense and sparse-direct, Hermitian and non-Hermitian DPP samplers.
\end{abstract}

%% file: intro.tex
\section{Introduction}
Determinantal Point Processes (DPPs) were first studied as a distinct class by
Macchi in the mid 1970's \cite{Macchi-1975, Macchi-1977} as a probability
distribution for the locations of repulsive (fermionic) particles, in direct
contrast with Permanental -- or, bosonic -- Point Processes. Particular instances of DPPs, representing the eigenvalue distributions of classical random
matrix ensembles, appeared in a series of papers in the beginning of the
1960's~\cite{MehtaGaudin-1960, Dyson-1962a, Dyson-1962b, Dyson-1962c, Ginibre-1965} (see \cite{Diaconis-2003} for a review, and \cite{EdelmanRao-2005} for a
computational perspective on sampling classical eigenvalue distributions). Some of the early investigations of (finite) DPPs involved their usage for uniformly
sampling spanning trees of
graphs~\cite{BurtonPemantle-1993, BenjaminiEtAl-2001} (see
\cref{fig:grid_40,fig:hexagonal_10}), non-intersecting random
walks~\cite{Johansson-2004}, and domino tilings of the Aztec diamond
\cite{Kasteleyn-1961, TemperleyFisher-1961, Kasteleyn-1963, ChhitaEtAl-2015} 
(see \cref{fig:aztec_10,fig:aztec_80}).

Let us recall that the class of {\em immanants} provides a generalization of
the determinant and permanent of a matrix by using an arbitrary character
$\chi$ of the symmetric group $S_n$ to determine the signs of terms in the
summation:
\[
  |A|^{\chi} \equiv \sum_{\sigma \in S_n} \chi(\sigma) \prod_{i=1}^n A_{i,\sigma(i)}.
\]
Choosing the trivial character $\chi \equiv 1$ yields the permanent, while
choosing $\chi(\sigma) = \text{sign}(\sigma)$ provides the determinant.
The resulting generalization from Determinantal Point Processes and
Permanental Point Processes to Immanantal Point Processes was studied in
\cite{DiaconisEvans-2000}.

\input{figures/grid_40}
\input{figures/hexagonal_10}

\input{figures/aztec_10}
\input{figures/aztec_80}

\begin{definition}
A finite {\bf Determinantal Point Process} is a random variable
$\mathbf{Y} \sim \text{DPP}(K)$ over the power set of a ground set $\{0, 1, ..., n-1\} = [n]$ such that
\[
  \mathbb{P}[Y \subseteq \mathbf{Y}] = \det(K_Y),
\]
where $K \in \mathbb{C}^{n \times n}$ is called the
{\bf marginal kernel matrix} and $K_Y$ denotes the restriction of $K$ to the
row and column indices of $Y$.
\end{definition}

We can immediately observe that the $j$-th diagonal entry of a marginal kernel
$K$ is the probability of index $j$ being in the sample, $\mathbf{Y}$, so the
diagonal of every marginal kernel must lie in
$[0, 1] \subset \mathbb{R}$. A characteristic requirement for a complex matrix
to be admissible as a marginal kernel is given in the following proposition,
due to Brunel~\cite{Brunel-2018}, which we will provide an alternative proof of
after introducing our factorization-based sampling algorithm.

\begin{prop}[Brunel \cite{Brunel-2018}]
A matrix $K \in \mathbb{C}^{n \times n}$ is admissible as a DPP marginal
kernel iff
\[
  (-1)^{|J|} \det(K - \mathds{1}_J) \ge 0,\;\;\forall J \subseteq [n].
\]
\end{prop}
\noindent
We can also observe that, because determinants are preserved under similarity
transformations, there is a nontrivial equivalence class for marginal kernels
-- which is to say, there are many marginal kernels defining the same DPP:
\begin{prop}
The equivalence class of a DPP kernel $K \in \mathbb{C}^{n \times n}$ contains
its orbit under the group of diagonal similarity transformations, i.e.,
\[
  \{ D^{-1} K D \, : \, D = \text{diag}(d), d \in (\mathbb{C} \setminus \{0\})^n \}.
\]
If we restrict to the classes of complex Hermitian or real symmetric kernels,
the same statement holds with the entries of $d$ restricted to the circle
group $U(1) = \{z \in (\mathbb{C},\times) : |z| = 1 \}$ and the scalar
orthogonal group $O(1) = \{z \in (\mathbb{R},\times) : |z| = 1\}$, respectively.
\end{prop}
\begin{proof}
Determinants are preserved under similarity transformations, so the probability
of inclusion of each subset is unchanged by global diagonal similarity. That
a unitary diagonal similarity preserves Hermiticity follows from recognition
that it becomes a Hermitian congruence. Likewise, a signature matrix similarity
transformation becomes a real symmetric congruence.
\end{proof}
\noindent

In most works, with the notable exception of Kasteleyn matrix approaches to
studying domino tilings of the Aztec
diamond~\cite{Kasteleyn-1961, TemperleyFisher-1961, Kasteleyn-1963,
ChhitaEtAl-2015}, the marginal kernel is assumed Hermitian. And
Macchi~\cite{Macchi-1975} showed (Cf.~\cite{Soshnikov-2000, HoughEtAl-2006})
that a Hermitian matrix is admissibile as a marginal kernel if and only if its
eigenvalues all lie in $[0, 1]$. The viewpoint of these eigenvalues as
probabilities turns out to be productive, as the most common sampling algorithm
for Hermitian DPPs, due to~\cite{HoughEtAl-2006} and popularized in the
machine learning community by~\cite{KuleszaTaskar-2012}, produces an equivalent
random orthogonal projection matrix by preserving eigenvectors with probability 
equal to their eigenvalue:

\begin{theorem}[Theorem 7 of \cite{HoughEtAl-2006}]
Given any Hermitian marginal kernel matrix $K \in \mathbb{C}^{n \times n}$ with
spectral decomposition $Q \Lambda Q^H$, sampling from
$\mathbf{Y} \sim \text{DPP}(K)$ is equivalent to sampling from a realization of
the random DPP with kernel $Q_{:,\mathbf{Z}} Q_{:,\mathbf{Z}}^H$, where
$Q_{:,\mathbf{Z}}$ consists of the columns of $Q$ with
$\mathbb{P}[j \in \mathbf{Z}] = \Lambda_j$, independently for each $j$.
\end{theorem}

Such an orthogonal projection marginal kernel is said to define a
{\bf Determinantal Projection Process}~\cite{HoughEtAl-2006}, or
{\bf elementary DPP}~\cite{KuleszaTaskar-2012},
and it is known that the resulting samples almost surely have cardinality equal
to the rank of the projection (e.g., Lemma 17 of \cite{HoughEtAl-2006}).
The algebraic specification in Alg.~1 of \cite{KuleszaTaskar-2012} of the
Alg.~18 of \cite{HoughEtAl-2006} involved an $O(n k^3)$ approach, where $k$ is
the rank of the projection kernel. In many important cases, such as uniformly
sampling spanning trees or domino tilings, $k$ is a large fraction of $n$, and
the algorithm has quartic complexity (assuming standard matrix multiplication).
In the words of \cite{KuleszaTaskar-2012}:
\begin{quote}
Alg. 1 runs in time $O(n k^3)$, where $k$ is the number of eigenvectors
selected [...] the  initial  eigendecomposition [...] is often the
computational bottleneck, requiring $O(n^3)$ time. Modern multi-core machines
can compute eigendecompositions up to $n\approx 1,000$ at interactive speeds of 
a few seconds, or larger problems up to $n\approx 10,000$ in around ten
minutes.
\end{quote}

A faster, $O(n k^2)$ algorithm for sampling elementary DPPs was given as
Alg.~2 of \cite{Gillenwater-2014}. We will later show that this approach is
equivalent to a small modification of a diagonally-pivoted, rank-revealing,
left-looking, Cholesky factorization~\cite{Higham-1990}, where the pivot index
is chosen at each iteration by sampling from the probability distribution
implied by the diagonal of the remaining submatrix (which is maintained
out-of-place).

As a brief aside, we recall that the distinction between {\em up-looking},
{\em left-looking},
and {\em right-looking} factorizations is based upon their choice of loop
invariant. Given the matrix partitioning
\[
  \begin{pmatrix} A_{1,1} & A_{1,2} \\
  A_{2,1} & A_{2,2}
  \end{pmatrix},
\]
where the diagonal of $A_{1,1}$ comprises the list of eliminated pivots:
an {\em up-looking} algorithm will have only updated $A_{1,1}$ (by overwriting
it with its factorization), a {\em left-looking} algorithm will also have
overwritten $A_{2,1}$ with the corresponding block column of the lower
triangular factor
(and $A_{1,2}$ with its upper-triangular factor in the nonsymmetric case), and a
{\em right-looking} algorithm will additionally have overwritten $A_{2,2}$
with the Schur complement resulting from the block elimination of $A_{1,1}$.

Researchers have begun proposing algorithms which directly sample from
Hermitian marginal kernels by sequentially deciding whether each index should be
in the result by sampling the Bernoulli process defined by the probability of
the item's inclusion, conditioned on the combination of all of the explicit
inclusion and exclusion decisions made so far. For example, when deciding whether to include index $j$ in the sample, we will have already partitioned $[0,...,j-1]$ into a set of included indices, $F$, and excluded indices, $G$. So we must
sample index $j$ with probability $\mathbb{P}[j \in \mathbf{Y} \,|\, F \subseteq \mathbf{Y},\, G \cap \mathbf{Y} = \emptyset]$.
Directly performing such sampling using formulae for marginal kernels of
conditional DPPs is referred to in \cite{Launay-2018} as {\em
sequential sampling} (Cf. \cite{ChenEtAl-2018} for greedy, maximum-likelihood
inference).

The primary contribution of this manuscript is to show that sequential sampling
can be performed via a small modification of an unpivoted $LDL^H$
factorization process (where $L$ is unit lower-triangular and $D$ is real
diagonal) and to extend them to non-Hermitian
marginal kernels using an unpivoted $LU$ factorization;
\cite{GartrellEtAl-2018} motivates an algorithm for learning non-Hermitian DPP
kernels by their ability to incorporate both
attraction and repulsion between items.

It is then demonstrated that high-performance factorizations
techniques~\cite{AndersonEtAl-1990,ChanEtAl-2008,ButtariEtAl-2009}
lead to orders of magnitude accelerations, and that sparse-direct
techniques~\cite{Schreiber-1982,AshcraftGrimes-1989,ChenEtAl-2008} can yield
further orders of magnitude speedups.

%% file: figures/grid_40.tex
\begin{figure}

\begin{subfigure}{0.5\textwidth}
\centering
\includegraphics[width=1.75in]{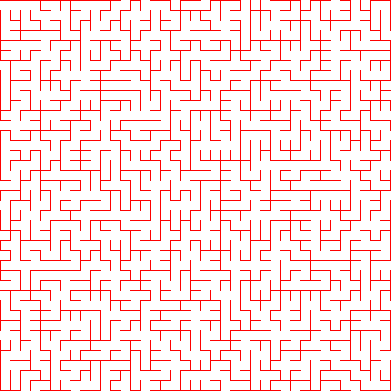}
\end{subfigure}
\begin{subfigure}{0.5\textwidth}
\centering
\includegraphics[width=1.75in]{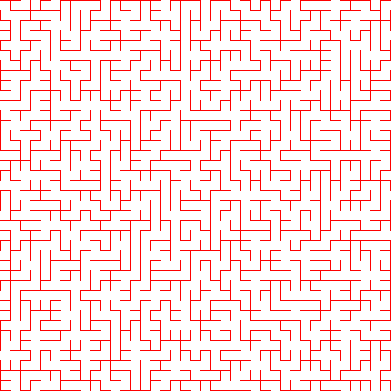}
\end{subfigure}

\caption{Two samples from the uniform distribution over spanning trees of a 40
x 40 box of $\mathbb{Z}^2$. Each has likelihood $\text{exp}(-1794.24)$.}
\label{fig:grid_40}

\end{figure}

%% file: figures/hexagonal_10.tex
\begin{figure}

\begin{subfigure}{0.5\textwidth}
\centering
\includegraphics[width=1.75in]{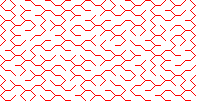}
\end{subfigure}
\begin{subfigure}{0.5\textwidth}
\centering
\includegraphics[width=1.75in]{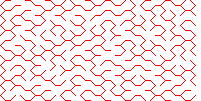}
\end{subfigure}

\caption{Two samples from the uniform distribution over spanning trees of a 10
x 10 section of hexagonal tiles. Each has likelihood $\text{exp}(-299.101)$.}
\label{fig:hexagonal_10}

\end{figure}

%% file: figures/aztec_10.tex
\begin{figure}

\begin{subfigure}{0.5\textwidth}
\centering
\includegraphics[width=1.75in]{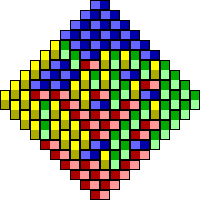}
\end{subfigure}
\begin{subfigure}{0.5\textwidth}
\centering
\includegraphics[width=1.75in]{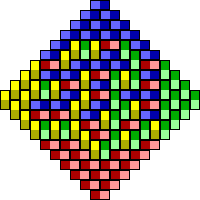}
\end{subfigure}

\caption{Two samples from the Aztec diamond DPP of order 10, defined via a
complex, non-Hermitian kernel based upon Kenyon's formula over the Kasteleyn
matrix. The tiles all begin on dark squares and are oriented either:
left (blue), right (red), up (yellow), or down (green).
Their likelihoods are both $\text{exp}(-38.1231)$.}
\label{fig:aztec_10}

\end{figure}

%% file: figures/aztec_80.tex
\begin{figure}

\begin{subfigure}{0.5\textwidth}
\centering
\includegraphics[width=2.4in]{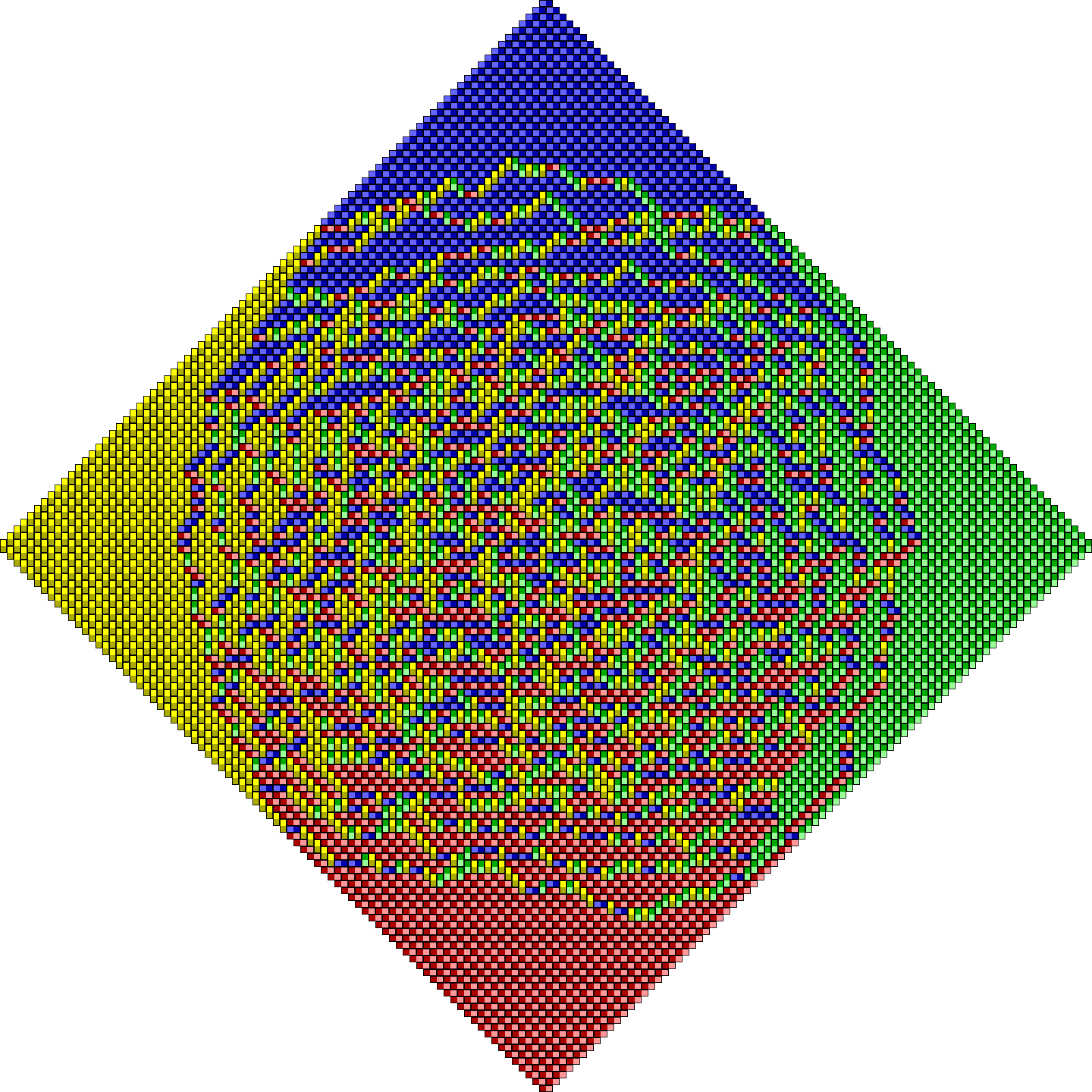}
\end{subfigure}
\begin{subfigure}{0.5\textwidth}
\centering
\includegraphics[width=2.4in]{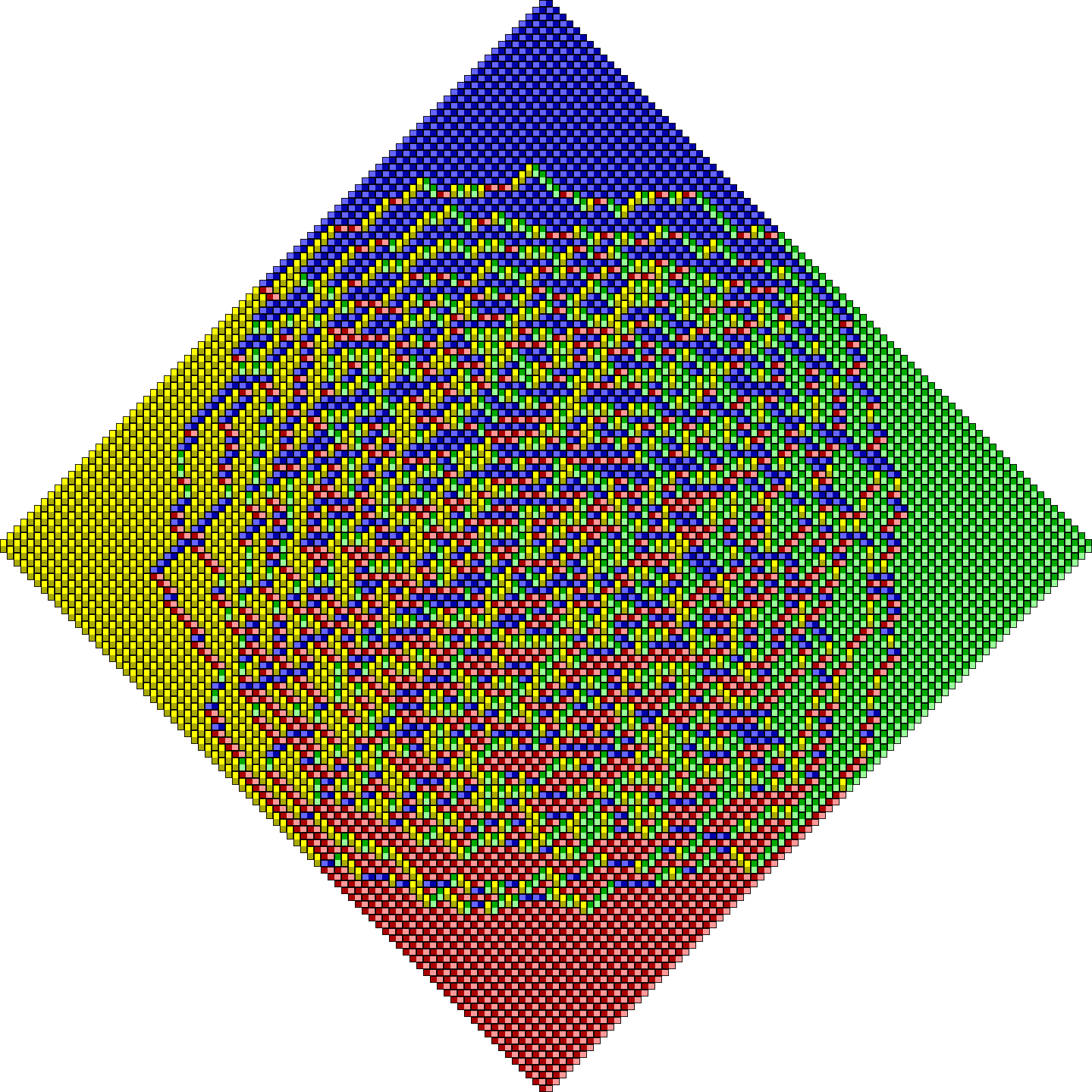}
\end{subfigure}

\caption{Two samples from the Aztec diamond DPP of order 80, defined via a
complex, non-Hermitian kernel based upon Kenyon's formula over the Kasteleyn
matrix. Their likelihoods are both $\text{exp}(-2245.8)$.
The diamond is large enough to clearly display the {\em arctic circle} phenomenon discussed in \cite{JockuschEtAl-1998} and \cite{ChhitaEtAl-2015}: ``when $n$
is sufficiently large, the shape of the central sub-region becomes arbitrarily
close to a perfect circle of radius $n/\sqrt{2}$ for all but a negligible
proportion of the tilings.''}
\label{fig:aztec_80}

\end{figure}

%% file: prototype_factorization.tex
\section{Prototype factorization-based DPP sampling}
To describe factorization processes, we will extend our earlier notation that,
for an $n \times n$ matrix $K$ and an index subset $Y \subseteq [n]$, $K_Y$
refers to the restriction of $K$ to the row and column indices of $Y$. Given
a second index subset $Z \subseteq [n]$, $K_{Y, Z}$ will denote the restriction
of $K$ to the $|Y| \times |Z|$ submatrix consisting of the rows indices in $Y$
and the column indices in $Z$. And we use the notation $[j:k]$ to represent the
integer range $\{j, j+1, ..., k-1\}$.

Our derivation will make use of a few elementary propositions on the forms of
marginal kernels for conditional DPPs. These propositions will allow us to
define modifications to the pivots of an LU factorization so that the resulting
Schur complements correspond to the marginal kernel of the DPP over the
remaining indices, conditioned on the inclusion decisions of the indices
corresponding to the eliminated pivots.

\begin{prop}
\label{condition_on_inclusion}
Given disjoint subsets $A,B \subseteq [n]$ of the ground set of a DPP with
marginal kernel $K$, almost surely
\[
  \mathbb{P}[B \subseteq \mathbf{Y} | A \subseteq \mathbf{Y}] =
  \det(K_B - K_{B,A} K_A^{-1} K_{A,B}).
\]
\end{prop}
\begin{proof}
If $A \subseteq \mathbf{Y}$, then
$\det(K_A) = \mathbb{P}[A \subseteq \mathbf{Y}] > 0$ almost surely, so we may
perform a two-by-two block LU decomposition
\[
\begin{pmatrix} K_A & K_{A,B} \\ K_{B,A} & K_B \end{pmatrix} =
\begin{pmatrix} I & 0 \\ K_{B,A} K_{A}^{-1} & K_B - K_{B,A} K_A^{-1} K_{A,B}
\end{pmatrix}
\begin{pmatrix} K_A & K_{A,B} \\ 0 & I \end{pmatrix}.
\]
That $\det : GL(n,\mathbb{C}) \mapsto (\mathbb{C},\times)$ is a homomorphism
yields
\[
  \det(K_{A \cup B}) = \det(K_A) \det(K_B - K_{B,A} K_A^{-1} K_{A,B}).
\]
The result then follows from the definition of conditional probabilities for a
DPP:
\[
  \mathbb{P}[B \subseteq \mathbf{Y} | A \subseteq \mathbf{Y}] =
  \frac{\mathbb{P}[A,B \subseteq \mathbf{Y}]}{\mathbb{P}[A \subseteq \mathbf{Y}]} =
  \frac{\det(K_{A \cup B})}{\det(K_A)}.
\]
\end{proof}

\begin{prop}
\label{condition_on_element_exclusion}
Given disjoint $a \in [n]$ and $B \subset [n]$ for a ground set $[n]$ of a DPP
with marginal kernel $K$, almost surely
\[
  \mathbb{P}[B \subset \mathbf{Y} | a \notin \mathbf{Y}] =
  \det(K_B - K_{B,a} (K_a - 1)^{-1} K_{a,B}).
\]
\end{prop}
\begin{proof}
\begin{align*}
  \mathbb{P}[B \subset \mathbf{Y} | a \notin \mathbf{Y}] &= 
    \frac{\mathbb{P}[a \notin \mathbf{Y} | B \subset \mathbf{Y}] \mathbb{P}[B \subset \mathbf{Y}]}{\mathbb{P}[a \notin \mathbf{Y}]} \\
  &= \frac{(1 - \mathbb{P}[a \in \mathbf{Y} | B \subset \mathbf{Y}]) \mathbb{P}[B \subset \mathbf{Y}]}{1 - \mathbb{P}[a \in \mathbf{Y}]} \\
  &= \frac{\det(K_B) - \det(K_{a \cup B})}{1 - K_a} \\
  &= \det(K_B)(1 - K_{a,B} \frac{{K_B}^{-1}}{K_a - 1} K_{B,a}) \\
  &= \det(K_B - K_{B,a} (K_a - 1)^{-1} K_{a,B}),
\end{align*}
where the last equality makes use of the Matrix Determinant Lemma. The formulae
are well-defined almost surely.
\end{proof}

Propositions \ref{condition_on_inclusion} and \ref{condition_on_element_exclusion} are
enough to derive our direct, non-Hermitian DPP sampling algorithm. But, for
the sake of symmetry with Proposition \ref{condition_on_inclusion}, we first
generalize to set exclusion:

\begin{prop}
\label{condition_on_exclusion}
Given disjoint subsets $A,B \subset \mathcal{Y}$, almost surely
\[
  \mathbb{P}[B \subseteq \mathbf{Y} | A \subseteq \mathbf{Y}^c] =
  \det(K_B - K_{B,A} (K_A - I)^{-1} K_{A,B}).
\]
\end{prop}
\begin{proof}
The claim follows from recursive formulation of conditional marginal kernels using the previous proposition. The resulting kernel is equivalent to the Schur
complement produced from the block LU factorization
\[
\begin{pmatrix} K_A - I & K_{A,B} \\ K_{B,A} & K_B \end{pmatrix} =
\begin{pmatrix} I & 0 \\
  K_{B,A} (K_A - I)^{-1} & K_B - K_{B,A} (K_A - I)^{-1} K_{A,B} \end{pmatrix}
\begin{pmatrix} K_A - I & K_{A,B} \\ 0 & I \end{pmatrix},
\]
as the subtraction of 1 from each eliminated pivot commutes with the outer
product updates.
\end{proof}

\begin{theorem}[Factorization-based DPP sampling]
\label{factorization-dpp}
Given a (possibly non-Hermitian) marginal kernel matrix $K$ of order $n$,
an LU factorization of $K$ can be modified as in Algorithm
\ref{lst:unblocked_dpp} to almost surely provide a sample from $\text{DPP}(K)$.
This algorithm involves roughly $\frac{2}{3} n^3$ floating-point operations,
and the likelihood of any returned sample will be given by the product of the
absolute value of the diagonal of the result. If the marginal kernel is
Hermitian, the work can be roughly halved by exploiting symmetry in the Schur
complement updates.
\end{theorem}
\begin{proof}
We first demonstrate, by induction, that our factorization algorithm samples
the DPP generated by the marginal kernel $K$. The loop invariant is that, at
the start of the iteration for pivot index $j$, $A_{[j:n]}$ represents
the equivalence class of kernels for the DPP over indices $[j:n]$ conditioned on
the inclusion decisions for indices $0,...,j-1$.

Since the diagonal entries of a kernel matrix represent the likelihood of the
corresponding index being in the sample, the loop invariant implies that index
$j$ is kept with the correct conditional probability.
Prop.\ \ref{condition_on_inclusion}
shows that the loop invariant is almost surely maintained when the Bernoulli
draw is successful, and Prop.\ \ref{condition_on_element_exclusion} handles
the alternative.
Thus, the loop invariant holds almost surely, and, upon completion, the
proposed algorithm samples each subset with the correct probability by
sequentially iterating over each index, making an inclusion decision with the
appropriate conditional probability.

The likelihood of a sample produced by the algorithm is thus the product of the
likelihoods of the results of the Bernoulli draws: when a draw for a diagonal
entry $p_j$ is successful, its probability was $p_j$, and, when unsuccessful, $1 - p_j$. In both cases, the multiplicative contribution is the absolute value
of the final state of the $j$'th diagonal entry.
\end{proof}

\begin{figure}
\begin{mdframed}[backgroundcolor=light-gray, roundcorner=8pt, leftmargin=0, rightmargin=1, innerleftmargin=15, innertopmargin=0, innerbottommargin=0, outerlinewidth=1, linecolor=light-gray]
\begin{lstlisting}[frame=single, language=python, mathescape=true,
  label={lst:unblocked_dpp}, numbers=left,
  caption={Unblocked, right-looking, non-Hermitian DPP sampling.
  For a sample $Y$, the returned matrix $A$ will contain the in-place $LU$
  factorization of $K - \mathds{1}_{Y^C}$, where $\mathds{1}_{Y^C}$ is the
  diagonal indicator
  matrix for the entries not in the sample. Symmetry can be exploited in the
  outer products when the kernel is Hermitian.}]
sample := []; A := K
for j in range(n):
  sample.append(j) if Bernoulli($A_j$) else $A_j \,-\!\!=\, 1$
  $A_{[j+1:n],\,j} \,\,/\!\!=\,\, A_j$
  $A_{[j+1:n]} \,\,-\!\!=\,\, A_{[j+1:n],\,j}\,A_{j,\,[j+1:n]}$
return sample, A
\end{lstlisting}
\end{mdframed}
\end{figure}

Theorem \ref{factorization-dpp} provides us with another interpretation of the
generic DPP kernel admissibility condition of \cite{Brunel-2018}:
\begin{proof}[Proposition 1 above]
When Algorithm \ref{lst:unblocked_dpp} produces a sample $Y$, the resulting
upper and strictly-lower triangular portions of the resulting matrix
respectively contain the $U$ and strictly-lower portion of the unit-diagonal $L$
from the LU factorization of $K - \mathds{1}_{Y^C}$. And we have:
\[
  \det(K - \mathds{1}_{Y^C}) = \det(U) = \prod_{j=0}^{n-1} U_j =
  \prod_{j \in Y} p_j \prod_{j \in Y^C} (p_j - 1),
\]
where $p_j$ is the inclusion probability for index $j$, conditioned on the
inclusion decisions of indices $0, ..., j-1$.

One can inductively show, working backwards from the last pivot, that
$(-1)^{|Y^C|} \det(K - \mathds{1}_{Y^C})$ always being non-negative is
equivalent to all
potential pivots produced by our sampling algorithm, the set of conditional
inclusion probabilities, living in $[0, 1]$. Otherwise, there would exist an
index inclusion decision change which would not change the sign of
$\det(K - \mathds{1}_{Y^C})$.
\end{proof}

By replacing the Bernoulli samples of algorithm \ref{factorization-dpp} with
the maximum-likelihood result for each index inclusion, we arrive at an
analogous (approximate) maximum-likelihood inference algorithm.

In both cases, the same specializations that exist for modifying an unpivoted
LU factorization into a Cholesky or (unpivoted) $LDL^H$ or $LDL^T$
factorization apply to our DPP sampling algorithms. And as we will see in the next two sections, so do high-performance dense and sparse-direct factorization
techniques.

As a brief aside, for real matrices, assuming standard matrix multiplication
algorithms, the highest-order terms for the operation counts of dense Gaussian
Elimination and an MRRR-based~\cite{Dhillon-2006} Hermitian eigensolver are
$\frac{2}{3} n^3$ and $\frac{10}{3} n^3$~\cite{HendricksonJessupSmith-1999}.
But the coefficient for the Hermitian
eigensolver is misleadingly small, as the initial phase of a Hermitian
eigensolver traditionally involves a unitary reduction to Hermitian tridiagonal
form that, due to only modest potential for data reuse, executes significantly less efficiently than traditional dense factorizations. So-called Successive
Band Reduction techniques~\cite{BischofLangSun-2000,HaidarLtaiefDongarra-2011,BallardDemmelKnight-2015} were therefore introduced as a gambit for
trading higher operation counts for decreased data movement and -- as a consequence -- increased performance.

\begin{theorem}[Factorization-based elementary DPP sampling]
\label{factorization-elementary-dpp}
Given a marginal kernel matrix $K$ of order $n$ which is an orthogonal
projection of rank $k$, performing $k$ steps of diagonally-pivoted
Cholesky factorization, where each pivot index is chosen by sampling from the
diagonal as in Algorithm \ref{lst:unblocked_elementary_dpp}, is
equivalent to sampling from $\text{DPP}(K)$. This approach has
complexity $O(n k^2)$ and almost surely completes and returns a sample of
cardinality $k$; the likelihood of the resulting sample is the square of the
product of the first $k$ diagonal entries of the partially factored matrix.
\end{theorem}
\begin{proof}
Alg.~\ref{lst:unblocked_elementary_dpp} is, up to permutation, algebraically equivalent to the sampling phase of Alg.~2 of \cite{Gillenwater-2014}: for the
sake of simplicity, it assumes the Gramian is preformed rather than forming it
on the fly from the factor. That the probability of a cardinality $k$ sample
$Y$ is, almost surely, equal to the product of the squares of the first $k$
diagonal entries of the result follows from recognizing that the lower triangle
of the top-left $k \times k$ submatrix will be the Cholesky factor of $K_Y$,
and
\[
  \mathbb{P}[Y = \mathbf{Y}] = \mathbb{P}[Y \subseteq \mathbf{Y}] = \det(K_Y) =
  \det(L_Y L_Y^H) = \det(L_Y) \det(L_Y^H) = \prod_{j=0}^{k-1} A_j^2.
\]
\end{proof}

\begin{figure}
\begin{mdframed}[backgroundcolor=light-gray, roundcorner=8pt, leftmargin=0, rightmargin=1, innerleftmargin=25, innertopmargin=0, innerbottommargin=0, outerlinewidth=1, linecolor=light-gray]
\begin{lstlisting}[frame=single, language=python, mathescape=true,
  label={lst:unblocked_elementary_dpp}, numbers=left,
  caption={Unblocked, left-looking, diagonally-pivoted, Cholesky-based sampling
    of a Hermitian Determinantal Projection Process. For a sample $Y$, the
  returned matrix will contain the in-place Cholesky factorization of $K_{Y}$.
  The computational cost is $O(n k^2)$.}]
$A \,:\!=\, K$; $d \,:\!=\, \text{diag}(K)$; orig_indices := [0:n]
for j in range(k):
  # Sample pivot index and permute
  Draw index t from $[j:n]$ with probability $d_t / (k - j)$
  Perform Hermitian swap of indices j and t of $A$
  Swap positions j and t of orig_indices and $d$
  $A_j := \sqrt{d_j}$
  if j == k - 1:
    break
  # Form new column and update diagonal
  $A_{[j+1:n],\,j} \,\,-\!\!=\, A_{[j+1:n],\,[0:j]} A_{j,\,[0:j]}^H$
  for t in range(j+1, n):
    $A_{t, j} \,\,/\!\!=\,\, A_j$
    $d_t \,\,-\!\!=\,\, |A_{t, j}|^2$
return orig_indices[0:k], $A_{[0:k]}$
\end{lstlisting}
\end{mdframed}
\end{figure}

The permutations in Alg.~\ref{lst:unblocked_elementary_dpp} were introduced to
solidify the connection to a traditional diagonally-pivoted Cholesky
factorization. There is the additional benefit of simplifying the usage of Basic
Linear Algebra Subprograms
(BLAS)~\cite{LawsonEtAl-1979,DongarraEtAl-1988,DongarraEtAl-1990} calls for the
matrix/vector products. But the bulk of the work of this approach will be in
rank-one updates, which have essentially no data reuse, and are therefore
not performant on modern machines, where the peak floating-point performance
is substantially faster than what can be read directly from main memory.
The Linear Algebra PACKage (LAPACK)~\cite{AndersonEtAl-1990} was introduced in
1990 as proof that dense factorizations can be recast in terms of matrix/matrix
multiplications; we present analogues in the following section, as well as
multi-core, tiled extensions similar to
\cite{ChanEtAl-2008,ButtariEtAl-2009}.

%% file: hpc_dense_factorization.tex
\section{High-performance, dense, factorization-based DPP sampling}
The main idea of LAPACK~\cite{AndersonEtAl-1990} is to reorganize the
computations within dense linear algebra algorithms so that as much of the work
as possible is recast into composing matrices with nontrivial minimal
dimensions. Basic operations rich in such matrix compositions -- typically
referred to as {\em Level 3 BLAS}~\cite{DongarraEtAl-1990} -- are
the building blocks of LAPACK. In most cases, such {\em block sizes} range from
roughly 32 to 256, with 64 to 128 being most common.

In the case of triangular factorizations, such as Cholesky, $LDL^H$, and LU, a
{\em blocked} algorithm can be produced from the unblocked algorithm by
carefully matricizing each of the original operations. In the case of a
right-looking LU factorization without pivoting, one arrives at
Alg.~\ref{lst:blocked_lu}, where $A_{J_1}$ takes the place of the scalar pivot
and must be factored -- typically, using the unblocked algorithm being generalized -- before its components are used to solve against $A_{J_2,\,J_1}$ and
$A_{J_1,\,J_2}$.

\begin{mdframed}[backgroundcolor=light-gray, roundcorner=10pt, leftmargin=0, rightmargin=1, innerleftmargin=25, innertopmargin=0,innerbottommargin=5, outerlinewidth=1, linecolor=light-gray]
\begin{lstlisting}[frame=single, language=python, mathescape=true,
  label={lst:blocked_lu}, numbers=left,
  caption={Blocked LU factorization without pivoting. The functions triu
  and unit\_tril mirror MATLAB notation and respectively return the
  upper-triangular and unit-diagonal lower-triangular restrictions of their
  input matrices.}]
j := 0
while j < n:
  $\text{bsize} \,:=\, \min(\text{blocksize}, n - j)$
  $J_1 = [j:j+\text{bsize}]$; $J_2 = [j+\text{bsize}:n]$
  $A_{J_1} = \text{unblocked\_lu}(A_{J_1})$
  $A_{J_2,\,J_1} \,\,:=\,\, A_{J_2,\,J_1} \,\, \text{triu}(A_{J_1})^{-1}$
  $A_{J_1,\,J_2} \,\,:=\,\, \text{unit\_tril}(A_{J_1})^{-1} \,\, A_{J_1,\,J_2}$
  $A_{J_2} \,\,-\!\!=\,\, A_{J_2,\,J_1} \,\,\, A_{J_1,\,J_2}$
  $j \,\,+\!\!=\,\, \text{bsize}$
return sample, A
\end{lstlisting}
\end{mdframed}

The correct form of these solves can be derived via the relationship:
\begin{equation*}
\begin{pmatrix}A_{1} & A_{1,\,2} \\ A_{2,\,1} & A_{2}\end{pmatrix} =
\begin{pmatrix}L_{1} & 0 \\ L_{2,\,1} & L_{2}\end{pmatrix}
\begin{pmatrix}U_{1} & U_{1,\,2} \\ 0 & U_{2}\end{pmatrix} =
\begin{pmatrix}L_{1} U_{1} & L_{1} U_{1,\,2} \\
  L_{2,\,1} U_{1} & L_{2,\,1} U_{1,\,2} + L_{2} U_{2}
\end{pmatrix},
\end{equation*}
where we used shorthand of the form $A_{1,2}$ to represent $A_{J_1,\,J_2}$.
An unblocked factorization can be used to compute the triangular factors $L_1$
and $U_1$ of the diagonal block $A_1$, and then triangular solves of the form
$L_{2,1} := A_{2,1} U_1^{-1}$ and
$U_{1,2} := L_1^{-1} A_{1,2}$ yield the two panels. After forming the Schur
complement $S_2 := A_2 - L_{2,1} U_{1,2}$, the problem has been reduced to an
LU factorization with fewer variables -- that of $L_2 U_2 = S_2$.
Asymptotically, all of the work is performed in the outer-product updates which
form the Schur complements.

The conversion of Alg.~\ref{lst:unblocked_dpp}, an unblocked DPP sampler, into
Alg.~\ref{lst:blocked_dpp}, a blocked DPP sampler, is essentially identical to
the formulation of Alg.~\ref{lst:blocked_lu} from an unblocked LU factorization.
We emphasize that, while it is well-known that $LU$ factorizations without
pivoting fail on large classes of nonsingular matrices -- for example, any
matrix with a zero in the top-left position -- the analogue for DPP sampling
succeeds almost surely for any marginal kernel.

\begin{mdframed}[backgroundcolor=light-gray, roundcorner=10pt, leftmargin=0, rightmargin=1, innerleftmargin=25, innertopmargin=0,innerbottommargin=5, outerlinewidth=1, linecolor=light-gray]
\begin{lstlisting}[frame=single, language=python, mathescape=true,
  label={lst:blocked_dpp}, numbers=left,
  caption={Blocked, factorization-based, non-Hermitian DPP sampling.
  For a sample $Y$, the returned matrix $A$ will contain the in-place $LU$ factorization of
  $K - \mathds{1}_{Y^C}$, where $\mathds{1}_{Y^C}$ is the
  diagonal indicator matrix for the entries not in the sample, $\mathbf{Y}$.
  The usual
  specializations from $LU$ to $LDL^H$ factorization applies if the kernel is
  Hermitian.}]
sample := []; A := K; j := 0
Replace line 5 of main loop of Alg. 3 with:
  $\text{subsample}, A_{J_1} = \text{unblocked\_dpp}(A_{J_1})$
  sample.append(subsample + j)
return sample, A
\end{lstlisting}
\end{mdframed}

Beyond the order-of-magnitude improvement provided by such algorithms, even on
a single core of a modern computer, they also simplify the incorporation of
parallelism. Lifting algorithms into blocked form allows for each core to be
dynamically assigned tasks corresponding to individual updates of a
{\em tile}~\cite{ChanEtAl-2008,ButtariEtAl-2009}, a roughly
$\text{tile\_size} \times \text{tile\_size}$submatrix which, in our
experiments, was typically most performant for $\text{tile\_size} = 256$.

Within the context of Alg.~\ref{lst:blocked_dpp}, our tiled algorithm assigns
each \verb!unblocked_dpp! call an individual OpenMP
4.0~\cite{OpenMPBoard-2013} task which depends upon the last Schur complement
update of its tile. The triangular solves against $U_{J_1}$ are split into
submatrices mostly of the form $\text{tile\_size} \times \text{block\_size}$,
those of the solves against $L_{J_1}$ are roughly of the transpose dimensions,
while the Schur complement tasks are roughly of size
$\text{tile\_size} \times \text{tile\_size}$. In each case, tile reads and
writes are scheduled using dependencies on previous updates of the same tile.

The performance of such a dynamically-scheduled parallelization of the
Hermitian specialization of Alg.~\ref{lst:blocked_dpp} is demonstrated for
arbitrary dense, Hermitian marginal kernels on a 16-core Intel i9-7960x in
\cref{fig:real_ldl_dpp_perf}. The performance of a similarly parallelized
unpivoted $LDL^H$ factorization is similarly plotted, and one readily observes
that their runtimes are essentially identical.
The runtime of Alg.~\ref{lst:blocked_dpp} for arbitrary, complex, non-Hermitian
marginal kernels is similarly shown in \cref{fig:complex_lu_dpp_perf}. For purpose of comparison, we note that the double-precision High-Performance LINPACK
benchmark~\cite{DongarraEtAl-2002} achieves roughly 1 TFlop/second on this
machine.

\input{figures/real_ldl_dpp_perf}
\input{figures/complex_lu_dpp_perf}

It is worth emphasizing the compounding performance gains from both the
formulation of DPP sampling as a small modification of a level-3 BLAS focused
dense matrix factorization and from the 16-way parallelism. When combined, a
factor of 2500x speedup is observed relative to the timings on the same machine
of DPPy v0.1.0~\cite{GautierBardenetValko-2018} when sampling a spectrally-preprocessed arbitrary, dense Hermitian $5000 \times 5000$ matrix. Similar speedups
exist relative to the timings of both the ``sequentially thinned'' and spectrally-preprocessed algorithms of \cite{Launay-2018}.
In the case of DPPy, it would be fair to attribute at least an order of magnitude of the performance difference to its implementation being standard Python
code rather than optimized C++.

In the case of DPPs where the expected number of samples,
$k \equiv \text{trace}(K)$, is much less than the ground set size, $n$,
Derezinski et al.~\cite{DerezinskiEtAl-2019} have proposed a rejection-sampling
approach with a setup cost of $n \cdot \text{polylog}(n) \cdot \text{poly}(k)$
and a subsequent sampling cost of $\text{poly}(k)$. Ignoring polylogarithmic
terms,
their setup cost is $O(n k^6 + k^9)$ and the subsequent sampling cost is
$O(k^6)$. Large speedups can therefore be expected when $k \ll \sqrt{n}$.
But the examples in Figs.~1--4 of this article all involve $k$ being
a modest fraction of $n$, e.g., $k \approx n/4$ in the case of the Aztec
diamond, so the rejection sampler would involve a prohibitive $O(n^9)$ setup
and $O(n^6)$ subsequent sampling cost.

Timings for the $\mathbb{Z}^2$ and hexagonal-tiling Uniform Spanning Tree DPPs
can be extracted from \cref{fig:real_ldl_dpp_perf} via the formulae
$n \approx 2 d^2$ and $n \approx 6 d^2$, respectively, where $d$ is the order
of the grid. And timings for the uniform domino tilings can be recovered from
\cref{fig:complex_lu_dpp_perf} via the formula
$n \approx 4 d^2$. When computing results for the latter, it was noticed that
sampling domino tilings from the Kenyon formula DPP in single-precision, the
results are essentially always inconsistent once the diamond size exceeds 60,
but no inconsistencies have yet been observed with double-precision sampling.
See \cref{fig:corrupted_aztec_sample} for an example.

\input{figures/corrupted_sample}

Such a numerical instability in an unpivoted dense matrix factorization would 
lead any numerical analyst to wonder if dynamic pivoting can mitigate error
accumulation. Indeed, our factorization-based approach frees us to perform
arbitrary diagonal pivoting, as long as the pivot is decided before its
corresponding Bernoulli draw. While it will be the subject of future work,
the author conjectures that {\em maximum-entropy} pivoting, that is, pivoting
towards a diagonal entry with conditional probability as close to $\frac{1}{2}$
as possible, would be the most beneficial, as it maximizes the magnitude of
the smallest possible pivot.

Before moving on to sparse-direct DPP sampling, we demonstrate the multiple
orders of magnitude speedup that are possible for Determinantal Projection
Processes of sufficiently low rank. Unlike our other experiments,
\cref{fig:real_ldl_elementary_dpp_perf} only executes on a single core, as
the unblocked, left-looking rank-revealing elementary DPP sampling approach of
Alg.~\ref{lst:unblocked_elementary_dpp} spends the majority of its time in
matrix/vector multiplication. Studying the benefits of parallelizations of
elementary DPP samplers is left for future work.

\input{figures/real_ldl_elementary_dpp_perf}

%% file: figures/real_ldl_dpp_perf.tex
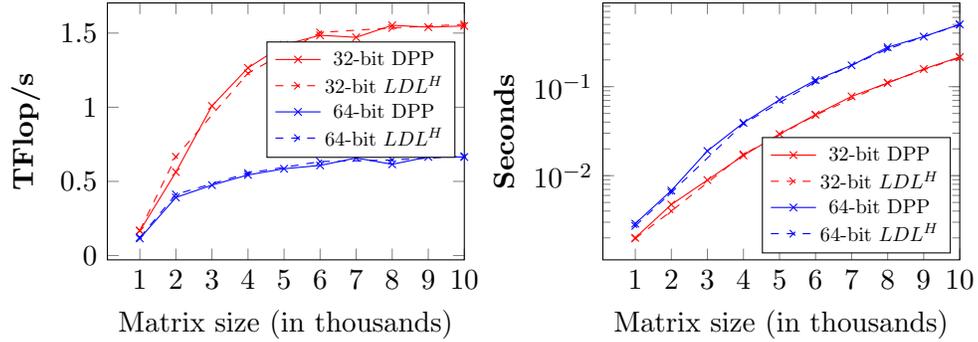
\begin{figure}
\begin{subfigure}{0.5\textwidth}
  \begin{tikzpicture}
  \begin{axis}[
    xlabel={Matrix size (in thousands)},
    ylabel={{\bf TFlop/s}},
    xtick=data,
    xmax=10,
    height=5cm,
    width=\columnwidth,
    legend style={at={(0.975,0.85)},anchor=north east,
      nodes={scale=0.7, transform shape}},
    ylabel near ticks
  ]




  \addplot [color=red,mark=x] coordinates {
    (1, 0.166484) 
    (2, 0.562889) 
    (3, 1.00892) 
    (4, 1.26533) 
    (5, 1.42087) 
    (6, 1.48505) 
    (7, 1.47142) 
    (8, 1.55245) 
    (9, 1.54009) 
    (10, 1.54802) 
  };
  \addlegendentry{32-bit DPP}




  \addplot [color=red,mark=x,dashed] coordinates {
    (1, 0.169156) 
    (2, 0.666026) 
    (4, 1.22837) 
    (6, 1.50421) 
    (8, 1.53484) 
    (10, 1.55903) 
  };
  \addlegendentry{32-bit $LDL^H$}




  \addplot [color=blue,mark=x] coordinates {
    (1, 0.115359) 
    (2, 0.391782) 
    (3, 0.474849) 
    (4, 0.543568) 
    (5, 0.585281) 
    (6, 0.607704) 
    (7, 0.657914) 
    (8, 0.61527) 
    (9, 0.665181) 
    (10, 0.664305) 
  };
  \addlegendentry{64-bit DPP}




  \addplot [color=blue,mark=x,dashed] coordinates {
    (1, 0.122346) 
    (2, 0.412119) 
    (4, 0.555966) 
    (6, 0.633747) 
    (8, 0.644287) 
    (10, 0.669698) 
  };
  \addlegendentry{64-bit $LDL^H$}

  \end{axis}
  \end{tikzpicture}
\end{subfigure}
\begin{subfigure}{0.5\textwidth}
  \begin{tikzpicture}
  \begin{semilogyaxis}[
    xlabel={Matrix size (in thousands)},
    ylabel={{\bf Seconds}},
    xtick=data,
    xmax=10,
    height=5cm,
    width=\columnwidth,
    legend style={at={(0.975, 0.025)},anchor=south east,
      nodes={scale=0.7, transform shape}},
    ylabel near ticks
  ]



  \addplot [color=red,mark=x] coordinates {
    (1, 0.0020022) 
    (2, 0.00473746) 
    (3, 0.00892042) 
    (4, 0.0168599) 
    (5, 0.0293247) 
    (6, 0.0484832) 
    (7, 0.0777028) 
    (8, 0.109934) 
    (9, 0.157783) 
    (10, 0.215329) 
  };
  \addlegendentry{32-bit DPP}




  \addplot [color=red,mark=x,dashed] coordinates {
    (1, 0.00197056) 
    (2, 0.00400384) 
    (4, 0.0173672) 
    (6, 0.0478657) 
    (8, 0.111195) 
    (10, 0.213809) 
  };
  \addlegendentry{32-bit $LDL^H$}



  \addplot [color=blue,mark=x] coordinates {
    (1, 0.00288953) 
    (2, 0.00680651) 
    (3, 0.0189534) 
    (4, 0.0392469) 
    (5, 0.0711909) 
    (6, 0.118479) 
    (7, 0.173781) 
    (8, 0.277385) 
    (9, 0.365314) 
    (10, 0.501777) 
  };
  \addlegendentry{64-bit DPP}




  \addplot [color=blue,mark=x,dashed] coordinates {
    (1, 0.0027245) 
    (2, 0.00647063) 
    (4, 0.0383716) 
    (6, 0.11361) 
    (8, 0.264892) 
    (10, 0.497737) 
  };
  \addlegendentry{64-bit $LDL^H$}

  \end{semilogyaxis}
  \end{tikzpicture}
\end{subfigure}

\caption{Dense, real $LDL^H$-based DPP sampling performance; tile sizes of
$128$ and $256$ were used for matrix sizes less than or equal to, and greater
than, $2000$, respectively.}

\label{fig:real_ldl_dpp_perf}

\end{figure}

%% file: figures/complex_lu_dpp_perf.tex
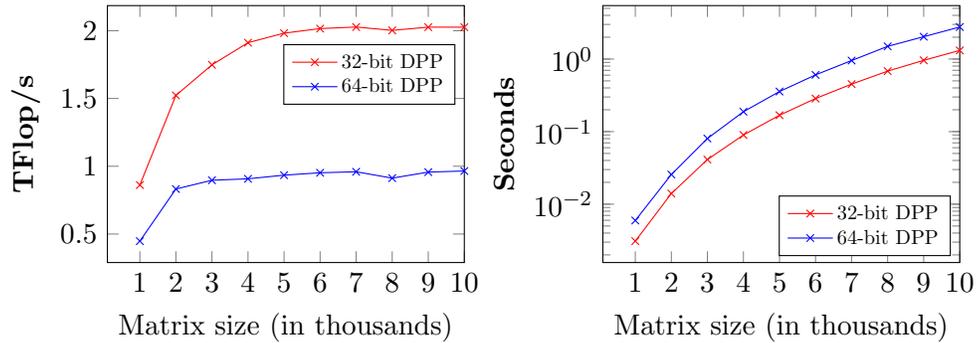
\begin{figure}
\begin{subfigure}{0.5\textwidth}
  \begin{tikzpicture}
  \begin{axis}[
    xlabel={Matrix size (in thousands)},
    ylabel={{\bf TFlop/s}},
    xtick=data,
    xmax=10,
    height=5cm,
    width=\columnwidth,
    legend style={at={(0.975,0.85)},anchor=north east,
      nodes={scale=0.7, transform shape}},
    ylabel near ticks
  ]

  \addplot [color=red,mark=x] coordinates {
    (1, 0.861485) 
    (2, 1.52274) 
    (3, 1.74821) 
    (4, 1.91163) 
    (5, 1.98197) 
    (6, 2.01573) 
    (7, 2.02731) 
    (8, 2.00249) 
    (9, 2.02633) 
    (10, 2.0261) 
  };
  \addlegendentry{32-bit DPP}

  \addplot [color=blue,mark=x] coordinates {
    (1, 0.447552) 
    (2, 0.832059) 
    (3, 0.896618) 
    (4, 0.907387) 
    (5, 0.934222) 
    (6, 0.951907) 
    (7, 0.959236) 
    (8, 0.912317) 
    (9, 0.956736) 
    (10, 0.965017) 
  };
  \addlegendentry{64-bit DPP}

  \end{axis}
  \end{tikzpicture}
\end{subfigure}
\begin{subfigure}{0.5\textwidth}
  \begin{tikzpicture}
  \begin{semilogyaxis}[
    xlabel={Matrix size (in thousands)},
    ylabel={{\bf Seconds}},
    xtick=data,
    xmax=10,
    height=5cm,
    width=\columnwidth,
    legend style={at={(0.975, 0.025)},anchor=south east,
      nodes={scale=0.7, transform shape}},
    ylabel near ticks
  ]
  \addplot [color=red,mark=x] coordinates {
    (1, 0.00309543) 
    (2, 0.0140098) 
    (3, 0.0411851) 
    (4, 0.0898054) 
    (5, 0.168183) 
    (6, 0.285752) 
    (7, 0.451172) 
    (8, 0.681818) 
    (9, 0.959368) 
    (10, 1.31616) 
  };
  \addlegendentry{32-bit DPP}

  \addplot [color=blue,mark=x] coordinates {
    (1, 0.00595834) 
    (2, 0.0256392) 
    (3, 0.0803017) 
    (4, 0.188086) 
    (5, 0.356803) 
    (6, 0.605101) 
    (7, 0.953536) 
    (8, 1.49656) 
    (9, 2.03191) 
    (10, 2.76334) 
  };
  \addlegendentry{64-bit DPP}

  \end{semilogyaxis}
  \end{tikzpicture}
\end{subfigure}

\caption{Dense, complex $LU$-based DPP sampling performance. For single-precision, a tile size of 128 was used up to matrix sizes of 3000, and tile sizes of 256 were used thereafter. For double-precision, the switch occurred above matrices of size 4000.}

\label{fig:complex_lu_dpp_perf}

\end{figure}

%% file: figures/corrupted_sample.tex
\begin{figure}
\centering
\includegraphics[width=2.8in]{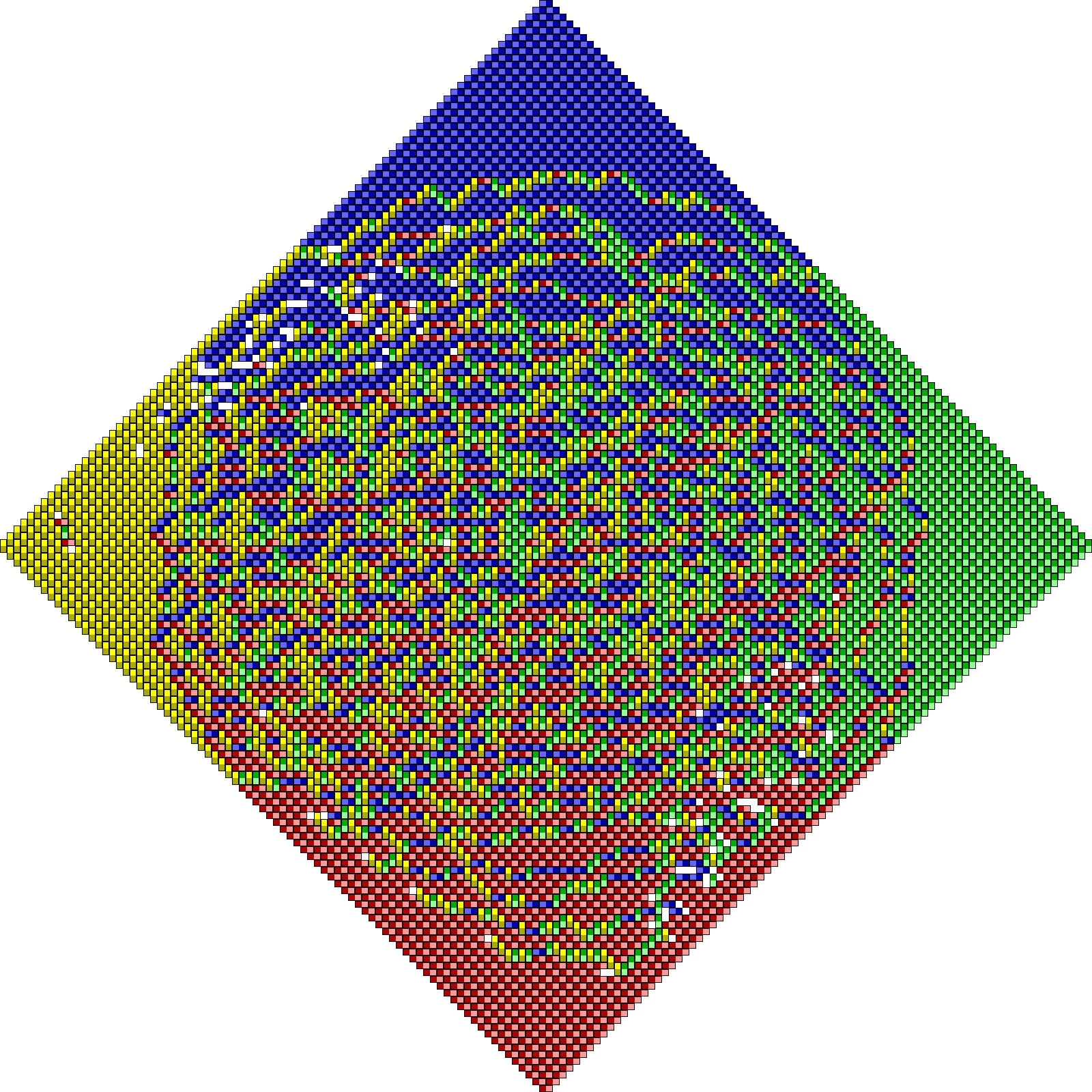} \\
\caption{Corrupted single-precision sample from an Aztec diamond of size 80.
Missing tiles are clearly visible in the image.}
\label{fig:corrupted_aztec_sample}
\end{figure}

%% file: figures/real_ldl_elementary_dpp_perf.tex
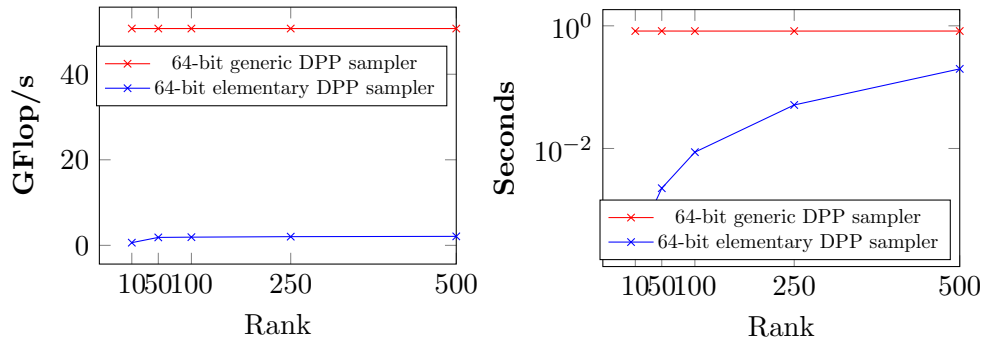
\begin{figure}
\begin{subfigure}{0.5\textwidth}
  \begin{tikzpicture}
  \begin{axis}[
    xlabel={Rank},
    ylabel={{\bf GFlop/s}},
    xtick=data,
    xmax=500,
    height=5cm,
    width=\columnwidth,
    legend style={at={(0.975,0.85)},anchor=north east,
      nodes={scale=0.7, transform shape}},
    ylabel near ticks
  ]

  \addplot [color=red,mark=x] coordinates {
    (10, 50.6829)
    (50, 50.6829)
    (100, 50.6829)
    (250, 50.6829)
    (500, 50.6829)
  };
  \addlegendentry{64-bit generic DPP sampler}

  \addplot [color=blue,mark=x] coordinates {
    (10, 0.628336)
    (50, 1.85286)
    (100, 1.91435)
    (250, 2.03423)
    (500, 2.09844)
  };
  \addlegendentry{64-bit elementary DPP sampler}

  \end{axis}
  \end{tikzpicture}
\end{subfigure}
\begin{subfigure}{0.5\textwidth}
  \begin{tikzpicture}
  \begin{semilogyaxis}[
    xlabel={Rank},
    ylabel={{\bf Seconds}},
    xtick=data,
    xmax=500,
    height=5cm,
    width=\columnwidth,
    legend style={at={(0.975, 0.025)},anchor=south east,
      nodes={scale=0.7, transform shape}},
    ylabel near ticks
  ]

  \addplot [color=red,mark=x] coordinates {
    (10, 0.822105)
    (50, 0.822105)
    (100, 0.82105)
    (250, 0.82105)
    (500, 0.822105)
  };
  \addlegendentry{64-bit generic DPP sampler}

  \addplot [color=blue,mark=x] coordinates {
    (10, 2.65251e-4)
    (50, 2.24877e-3)
    (100, 0.00870617)
    (250, 0.0512068)
    (500, 0.19856)
  };
  \addlegendentry{64-bit elementary DPP sampler}

  \end{semilogyaxis}
  \end{tikzpicture}
\end{subfigure}

\caption{Sequential, dense, real $LDL^H$-based elementary DPP sampling
performance for a ground set of size $5000$ and varying rank. Multithreaded
BLAS was explicitly disabled to ensure execution on a single core of the Intel
i9-7960x.}

\label{fig:real_ldl_elementary_dpp_perf}

\end{figure}

%% file: sparse_direct_factorization.tex
\section{Sparse-direct DPP sampling}
As was demonstrated in \cref{fig:real_ldl_dpp_perf}, the performance of a
generic Hermitian DPP sampler can be made to match that of a high-performance
Hermitian matrix factorization. This correspondence should not be a surprise,
as we have shown that matrix factorizations can be trivially modified to yield
Hermitian and non-Hermitian DPP samplers alike.
The same insight applies to the conversion of sparse-direct, unpivoted
$LDL^H$ factorization into sparse-direct Hermitian DPP samplers. Such an
implementation, mirroring many of the techniques from
CholMod~\cite{ChenEtAl-2008}, was implemented within Catamari.

The high-level approach is to use a dynamically scheduled -- via OpenMP 4.0
task scheduling -- {\em multifrontal
method}~\cite{DuffReid-1983,AshcraftGrimes-1989}
when the computational intensity of the factorization is deemed high enough
after determining the number of nonzeros and of the triangular factor and the
number of operations required to compute it. If the arithmetic count and
intensity is not sufficiently high, an {\em up-looking}, scalar
algorithm~\cite{ChenEtAl-2008} is used instead.

The multifrontal method performs the bulk of its work processing an
({\em elimination}) tree of dense {\em frontal matrices} in a post-ordering.
Each such frontal matrix for a supernode $s$ initially takes the form
\[
  \begin{pmatrix} A_s & * \\
    A_{\text{struct}(s),s} & 0
  \end{pmatrix},
\]
where $A_s$ is the diagonal block of the input matrix corresponding to supernode
$s$, $A_{\text{struct}(s),s}$ is the submatrix of $A$ below the supernodal
diagonal block that contains at least one nonzero in each row, and the $*$
denotes the upper-right quadrant not needing to be accessed due to Hermiticity.
Processing such a front involves adding any child Schur complements onto the
front, factoring $A_s = L_s D_s L_s^H$, solving against $D_s L_{s}^H$ to replace
$A_{\text{struct}(s),s}$ with $L_{\text{struct}(s),s}$, and replacing
the bottom-right quadrant with the Schur complement
$-L_{\text{struct}(s),s} D_s L_{\text{struct}(s),s}^H$.

This multifrontal process is then parallelized using a tile-based, dynamically
scheduled DPP sampler or dense factorization on the diagonal blocks, combined
with nested OpenMP tasks launched for each
(relaxed) supernode's~\cite{AshcraftGrimes-1989} subtree~\cite{HoggScott-2013}.


While the author is not aware of any well-studied sparse marginal kernel
matrices, \cite{MarietSra-2016} proposed a scheme for learning data-sparse
marginal kernels defined via Kronecker products and evaluated their technique
on the Amazon baby registry dataset of \cite{GillenwaterEtAl-2014}.
While DPP learning schemes are beyond the scope of this manuscript, the
dramatic factorization speedups that can be supplied by sparse-direct solvers
suggests the possibility of learning an entrywise sparse marginal kernel
matrix, perhaps via incorporating entrywise soft-thresholding into an iterative
learning scheme~\cite{GillenwaterEtAl-2014}.

We therefore implemented tandem sparse-direct factorization and sparse-direct
Hermitian DPP samplers to connect performance results of synthetic sparse DPPs
and discretized Partial Differential Equation solvers. Such a
statically-pivoted sparse-direct $LDL^H$ factorization is also
the critical computational kernel of primal-dual Interior Point Methods~\cite{AltmanGondzio-1993,Gondzio-1996}, where the first-order optimality conditions
are symmetric quasi-semidefinite~\cite{Vanderbei-1995,GeorgeEtAl-2000}.

Performance results for the dynamically scheduled sparse-direct solver
on a 3D Helmholtz equation discretized with trilinear, hexahedral elements~\cref{fig:helmholtz_3d_lens} are shown in \cref{fig:helmholtz_3d_perf}; results for
a 2D DPP analogue, where the marginal kernel takes the form of a scaled 2D
Laplacian, are shown in \cref{fig:laplacian_2d_0_72}.


\input{figures/helmholtz_3d_lens}

\input{figures/helmholtz_3d_perf}

\input{figures/laplacian_2d_0_72}

In the case of the 2D Laplacian sparse-direct DPP sampling over a
$200 \times 200$ grid, the timings were roughly 0.01 seconds on the 16-core
i9-7960x. Extrapolating from \cref{fig:real_ldl_dpp_perf} to a dense matrix size
of $40,000$, it is clear that several more orders of magnitude of efficiency
are gained with the sparse-direct formulation.

We close this section by noting that static pivoting should apply equally well
to non-Hermitian sparse-direct DPP sampling, as the nature of small pivots
being, by definition, rare, suggests a certain degree of stability. As in the
case of dense DPP sampling, probabilistic error analysis is warranted.

%% file: figures/helmholtz_3d_lens.tex
\begin{figure}
\centering
\includegraphics[width=3in]{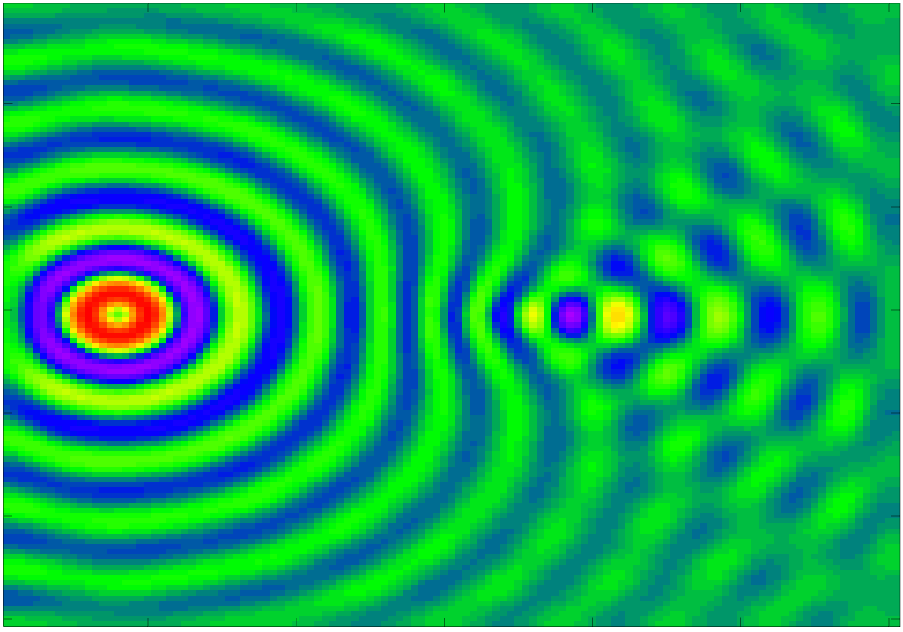}
\caption{A 2D slice for a solution of a 3D Helmholtz equation discretized with trilinear, hexahedra elements and PML boundary conditions for a Gaussian point source near the boundary and a converging lens near the center of the domain.}
\label{fig:helmholtz_3d_lens}
\end{figure}

%% file: figures/helmholtz_3d_perf.tex
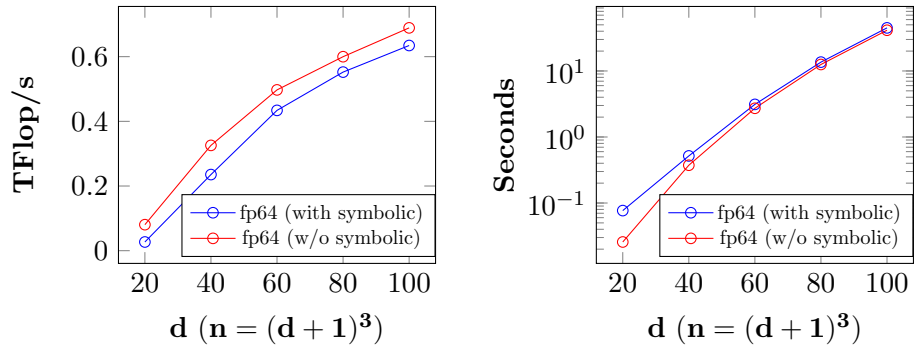
\begin{figure}
\begin{subfigure}{0.5\textwidth}
  \begin{tikzpicture}
  \begin{axis}[
    xlabel={$\mathbf{d}$ ($\mathbf{n=(d+1)^3}$)},
    ylabel={{\bf TFlop/s}},
    height=5cm,
    legend style={at={(1,0.025)},anchor=south east,
      nodes={scale=0.7, transform shape}}
  ]

    \addplot [color=blue,mark=o] coordinates {
      (20, 0.0269129)
      (40, 0.235388)
      (60, 0.433892) 
      (80, 0.552256)
      (100, 0.634487)
    };
    \addlegendentry{fp64 (with symbolic)}
    \addplot [color=red,mark=o] coordinates {
      (20, 0.080453) %
      (40, 0.325638) %
      (60, 0.497209) %
      (80, 0.599874) %
      (100, 0.689021) %
    };
    \addlegendentry{fp64 (w/o symbolic)}

  \end{axis}
  \end{tikzpicture}
\end{subfigure}
\begin{subfigure}{0.5\textwidth}
  \begin{tikzpicture}
  \begin{semilogyaxis}[
    xlabel={$\mathbf{d}$ ($\mathbf{n=(d+1)^3}$)},
    ylabel={{\bf Seconds}},
    height=5cm,
    legend style={at={(1,0.025)},anchor=south east,
      nodes={scale=0.7, transform shape}}
  ]

    \addplot [color=blue,mark=o] coordinates {
      (20, 0.0767901)
      (40, 0.516685)
      (60, 3.11796)
      (80, 13.6154)
      (100, 44.9421)
    };
    \addlegendentry{fp64 (with symbolic)}
    \addplot [color=red,mark=o] coordinates {
      (20, 0.0256876)
      (40, 0.373486)
      (60, 2.72091)
      (80, 12.5346)
      (100, 41.3851)
    };
    \addlegendentry{fp64 (w/o symbolic)}

  \end{semilogyaxis}
  \end{tikzpicture}
\end{subfigure}

\caption{Sparse-direct performance on trilinear, hexahedral discretization of
3D Helmholtz equation with Perfectly Matched Layer boundary conditions.}
\label{fig:helmholtz_3d_perf}
\end{figure}

%% file: figures/laplacian_2d_0_72.tex
\begin{figure}
\begin{subfigure}{0.5\textwidth}
\includegraphics[width=2in]{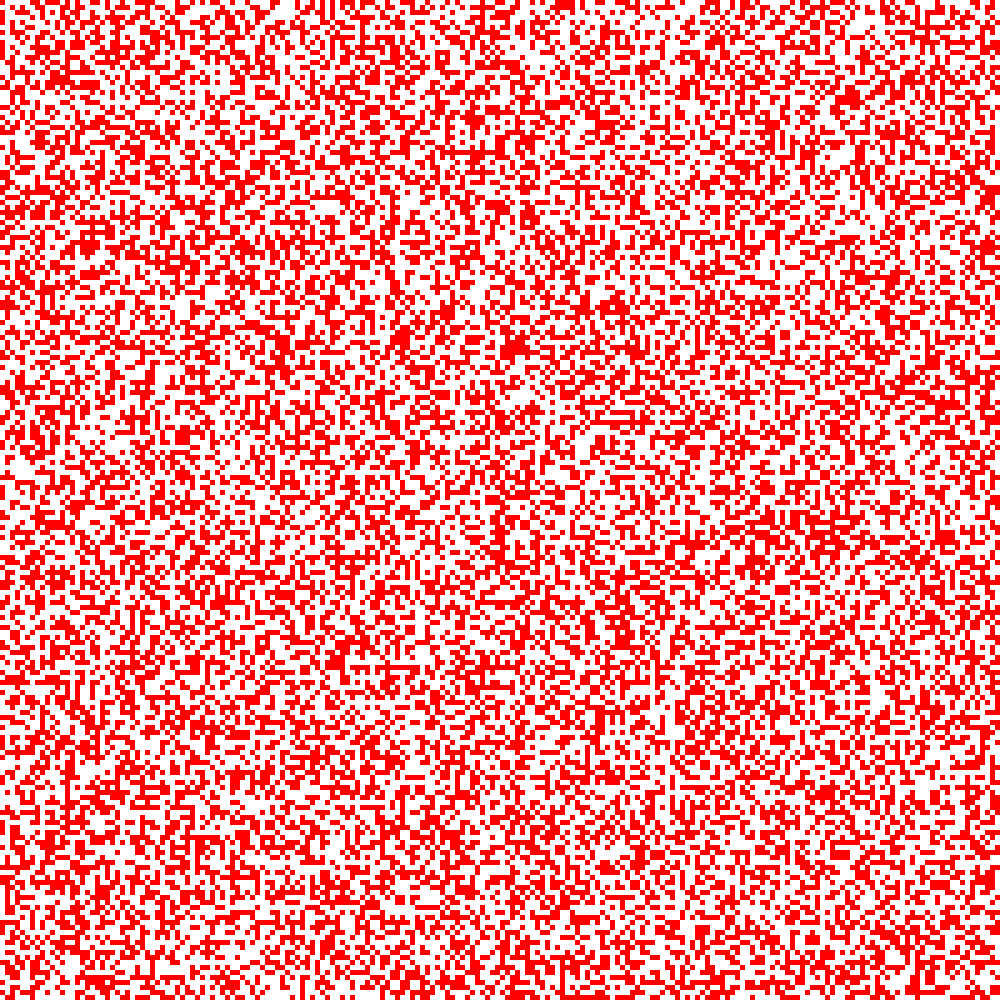}
\caption{Sample with log-likelihood of -27472.2.}
\end{subfigure}
\begin{subfigure}{0.5\textwidth}
\includegraphics[width=2in]{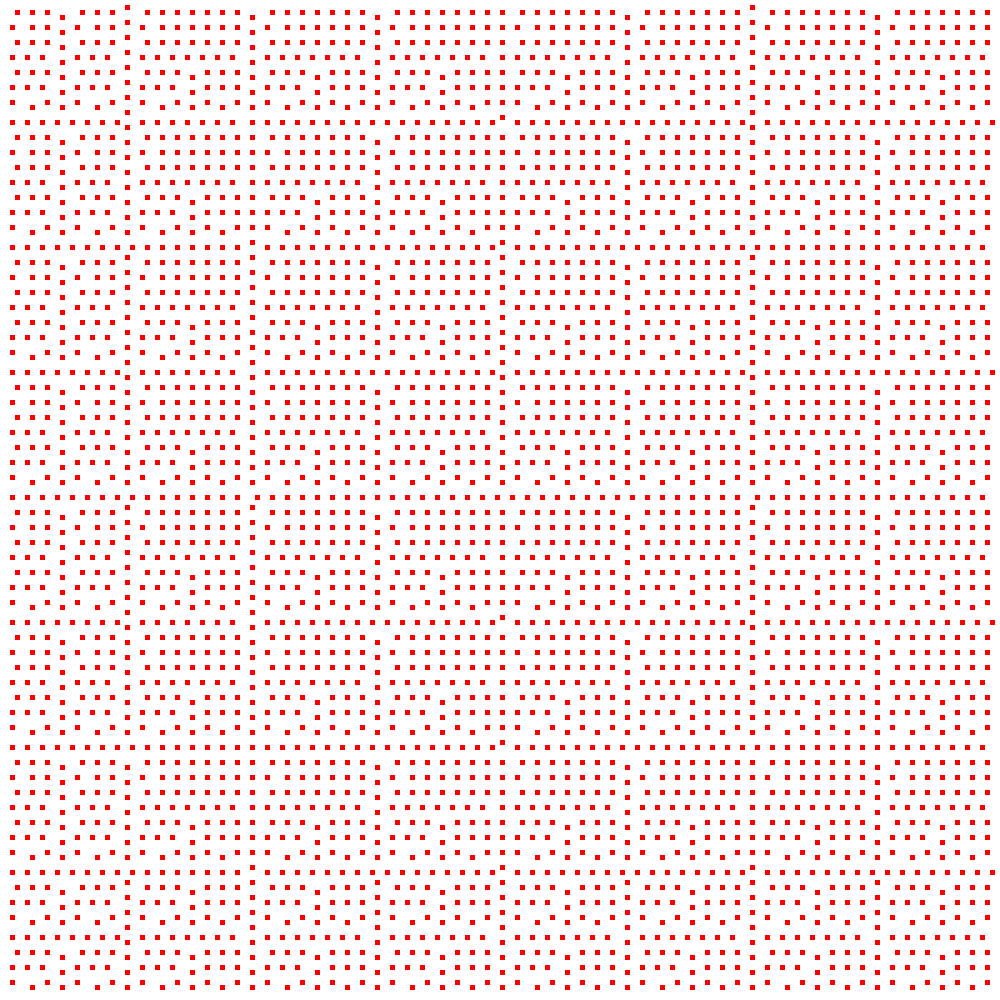}
\caption{Greedy maximum-likelihood sample with log-likelihood of -26058.}
\end{subfigure}
\caption{Samples from $-\sigma \Delta$ over a 200 x 200 grid,
where $\sigma=0.72$.}
\label{fig:laplacian_2d_0_72}
\end{figure}

%% file: conclusions.tex
\section{Conclusions}
A unified, generic framework for sequentially sampling both Hermitian and
non-Hermitian Determinantal Point Processes directly from their marginal kernel
matrices was presented. The prototype algorithm consisted of a small
modification of an unpivoted LU factorization -- in the Hermitian case, it
reduces to a modification of an unpivoted $LDL^H$ factorization -- and
high-performance, tiled algorithms were presented in the dense case, and an
analogue of CholMod~\cite{ChenEtAl-2008} was presented for the sparse case.
Further, it was shown that both greedy, maximum-likelihood sampling and
elementary DPP sampling can be understood and efficiently implemented using
small modifications of traditional matrix factorization techniques.

In addition to generalizing sampling algorithms from Hermitian to non-Hermitian
marginal kernels, the proposed approaches were implemented within the open
source, permissively licensed, Catamari package and shown to lead to orders
of magnitude speedups, even in the dense regime. The Hermitian sparse-direct
DPP sampler leads to further asymptotic speedups.
Future work includes theoretical and empirical exploration of the stability of
dynamic pivoting techniques, both for dense and sparse-direct
factorization-based DPP sampling, including the incorporation of
maximum-entropy pivoting for large instances of Kenyon-formula Aztec domino
tilings.

Further exploration of the performance tradeoffs between left-looking
and right-looking elementary DPP sampling, and at which rank it becomes
beneficial to use our generic sampling approach, is warranted -- especially on
multi-core and GPU-accelerated architectures. And, lastly, the
incorporation of a tiled extension of \cite{BientinesiEtAl-2008} for converting
a Hermitian L-ensemble kernel into a marginal kernel via $K = I - (L + I)^{-1}$
~\cite{Gillenwater-2014} is a natural extension of our proposed techniques and
could be expected to require the equivalent of 3 samples worth of time due to
the computational complexity.

But perhaps the most pressing future direction is to investigate the potential
of learning sparse marginal kernels on a benchmark problem, such as the
Amazon baby registry of \cite{GillenwaterEtAl-2014}. Ideally such a scheme
would expose a roughly block-diagonal covariance matrix which captures item
clusters and allows for sufficiently sparse triangular factors.

%% file: reproducibility.tex
\section*{Reproducibility}
The experiments in this paper were performed using version 0.2.5 of Catamari,
with Intel's Math Kernel Library BLAS on an Ubuntu 18.04 workstation with 64
GB of RAM and an Intel i9-7960x processor. Catamari can be downloaded from
\url{https://gitlab.com/hodge_star/catamari}.
\cref{fig:grid_40,fig:hexagonal_10} were generated using the TIFF output from
example driver \\
\verb!example/uniform_spanning_tree.cc!, while
\cref{fig:aztec_10,fig:aztec_80,fig:corrupted_aztec_sample} similarly used\\
\verb!example/aztec_diamond.cc!.

\cref{fig:real_ldl_dpp_perf,fig:complex_lu_dpp_perf} were
generated using data from \verb!example/dense_dpp.cc! and
\verb!example/dense_factorization.cc!; we emphasize that explicitly setting
both \verb!OMP_NUM_THREADS! and \verb!MKL_NUM_THREADS! to 16 on our 16-core
machine led to significant performance improvements (as opposed to the default
exploitation of hyperthreading). And \cref{fig:real_ldl_elementary_dpp_perf}
was produced with \\
\verb!example/dense_elementary_dpp.cc! -- with
\verb!OMP_NUM_THREADS! and \\
\verb!MKL_NUM_THREADS! explicitly set to one. The
algorithmic \verb!block_size! was also set to a value at least as large as the
rank so that a left-looking approach was used throughout.

\cref{fig:helmholtz_3d_lens,fig:helmholtz_3d_perf} were generated from
\verb!example/helmholtz_3d_pml.cc!, with the 2D-slice of the former taken
slightly away from the center plane. \cref{fig:laplacian_2d_0_72} was output
from the TIFF support of \\
\verb!example/dpp_shifted_2d_negative_laplacian.cc!.

%% file: acknowledgement.tex
\section*{Acknowledgements}
The author would like to sincerely thank Guillaume Gautier and R{\'{e}}mi
Bardenet for detailed comments on an early draft of this manuscript, as well as
their concise explanation of the correctness of the factorization-based sampler
at\\
\url{https://dppy.readthedocs.io/en/latest/finite_dpps/exact_sampling.html#id17}. I would also like to thank
Jennifer A. Gillenwater for suggesting the $I - (L + I)^{-1}$ conversion from
an L-ensemble and Alex Kulesza for answering some of my early DPP questions.
Lastly, I am thankful for the insightful comments and suggestions of the two
anonymous referees: they significantly improved the quality of the manuscript.